\documentclass[letterpaper,11pt]{amsart}

\usepackage[margin=1.3in]{geometry}
\usepackage{amsmath,amsthm,amssymb}
\usepackage{xspace,xcolor}
\usepackage[breaklinks,colorlinks,citecolor=teal,linkcolor=teal,urlcolor=teal,pagebackref,hyperindex]{hyperref}
\usepackage[alphabetic]{amsrefs}
\usepackage[all]{xy}

\theoremstyle{plain}
\newtheorem{thm}{Theorem}
\newtheorem*{thmA}{Theorem}

\newtheorem{lem}[thm]{Lemma}

\theoremstyle{definition}

\theoremstyle{remark}
\newtheorem{rmk}{Remark}

\def\C{{\mathbb C}}

\def\A{{\mathbb A}}
\def\P{{\mathbb P}}

\def\cH{\mathcal{H}}

\def\I{\mathcal{I}}
\def\J{\mathcal{J}}
\def\O{\mathcal{O}}

\def\fm{\mathfrak{m}}

\def\f{\phi}

\def\l{\lambda}

\def\m{\mu}
\def\om{\omega}

\def\s{\sigma}

\def\S{\Sigma}

\def\.{\cdot}
\def\^{\widehat}
\def\~{\widetilde}

\def\ov{\overline}

\def\rat{\dashrightarrow}
\def\surj{\twoheadrightarrow}
\def\inj{\hookrightarrow}

\def\({\left(}
\def\){\right)}

\renewcommand{\and}{ \ \ \text{ and } \ \ }

\DeclareMathOperator{\val} {val}

\DeclareMathOperator{\lct} {lct}

\DeclareMathOperator{\mld} {mld}
\DeclareMathOperator{\can} {can}

\begin{document}

\title{Erratum to: Birationally rigid hypersurfaces}

\author{Tommaso de Fernex}

\address{Department of Mathematics, University of Utah, 155 South 1400
East, Salt Lake City, UT 48112-0090, USA}
\email{{\tt defernex@math.utah.edu}}

\begin{abstract}
This note points out a gap in the proof of the main theorem of the article
\emph{Birationally rigid hypersurfaces} published in Invent.\ Math.\ {\bf 192} (2013), 533--566,
and provides a new proof of the theorem. 
\end{abstract}

\thanks{We thank J\'anos Koll\'ar for bringing to our attention the error in Lemma~9.2, and
Lawrence Ein, Mircea Musta\c t\u a, and Fumiaki Suzuki for useful discussions. 
The research was partially supported by 
NSF grant DMS-1402907 and NSF FRG grant DMS-1265285.}

\maketitle

The statements of Lemma~9.1 and Lemma~9.2 in \cite{dF13} are incorrect
(see the remarks at the end of this note for further comments on the errors).
The two lemmas are used in the proof of Theorem~A, which is the main theorem of \cite{dF13};
more precisely, they are used in the proof of Theorem~7.4 from which Theorem~A is deduced.
The error in Lemma~9.2 was brought to our attention by J\'anos Koll\'ar.

While Lemma~9.2 is not essential for the proof and can be circumvented, 
the gap left by the error in Lemma~9.1 appears to be more substantial because 
of the key role that Lemma~9.1 plays in the application of Theorem~B in the proof of Theorem~A.

In this note, we give a new argument to prove Theorem~A
which does not use Theorem~B. This does not fix, however, the proof of Theorem~7.4, 
which should therefore be considered unproven. 
We use the same notation and conventions as in \cite{dF13}.

\begin{thmA}[\protect{\cite{dF13}, Theorem~A}]
For $N \ge 4$, every smooth complex hypersurface $X \subset \P^N$ of degree $N$ is birationally superrigid. 
\end{thmA}

\begin{proof}
We assume that $N \ge 7$ and refer to \cite{dFEM03} for the remaining cases $4 \le N \le 6$. 

Suppose that $\f \colon X \rat X'$ is a birational map, but not an isomorphism, from $X$ 
to a Mori fiber space $X'$. The map is defined by a linear system $\cH$ whose members are
cut out by homogeneous forms of some degree $r$. 
Let $D,D' \in \cH$ be two general elements, and denote
\[
c := \can(X,D).
\]

Proposition~7.3 of \cite{dF13} implies that $c < 1/r$. On the other hand, 
Proposition~8.7 of \cite{dF13} implies that the set of points $Q \in X$ such that
$e_Q(D) > r$ is finite. 
It follows that the pair $(X,cD)$ is terminal in dimension one, and hence
there is a closed point $P \in X$ such that $\mld(P;X,cD) = 1$.  
This implies that $\mld(P;X,cD+P) \le 0$. 

Let $Y \subset X$ be a general hyperplane section 
through the point $P$, and let $B := D \cap D' \cap Y$. 
We remark that $Y$ is a smooth hypersurface of degree $N$ in $\P^{N-1}$
and $B$ is a complete intersection subscheme of $Y$ of codimension two. 
By inversion of adjunction (e.g., see Theorem~6.1 of \cite{dF13}),
we have $\mld(P;Y,cB) \le 0$. This means that $(Y,cB)$ is not log terminal near $P$. 
Notice, though, that $(Y,cB)$ is log terminal in dimension one. 
In fact, we have the following stronger property.

\begin{lem}\label{l:1}
The pair $(Y,2cB)$ is log terminal in dimension one.
\end{lem}

\begin{proof}
Let $C \subset Y$ be any irreducible curve.

Proposition~8.7 of \cite{dF13} implies that 
the set of points $Q \in X$ such that $e_Q(D \cap D') > r^2$
has dimension at most one. 
It follows by Proposition~8.5 of \cite{dF13} that, for a general choice of $Y$,
the set of points $Q \in Y$ such that $e_Q(B) > r^2$
is zero dimensional. Therefore we have 
$e_Q(B) \le r^2$ for a general point $Q \in C$. 

Fix such a point $Q \in C$, and let $S \subset Y$ be a smooth surface
cut out by general hyperplanes through $Q$. 
By Proposition~8.5 of \cite{dF13}, we have $e_Q(B|_S) \le r^2$. 
Since $B|_S$ is a zero-dimensional complete intersection subscheme of $S$, 
the multiplicity $e_Q(B|_S)$ is equal to
the Hilbert--Samuel multiplicity of the ideal $\I_{B|_S,Q} \subset \O_{S,Q}$
locally defining $B|_S$ near $Q$.
Then Theorem~0.1 of \cite{dFEM04}
implies that the log canonical threshold of $(S,B|_S)$ near $Q$ satisfies the inequality
\[
\lct_Q(S,B|_S) \ge \frac{2}{\sqrt{e_Q(B|_S)}}.
\]
Since $e_Q(B|_S) \le r^2$ and $c < 1/r$, this implies that $\lct_Q(S,B|_S) > 2c$, and hence
$(S,2cB|_S)$ is log terminal near $Q$. It follows by inversion of adjunction
that $(Y,2cB)$ is log terminal near $Q$. 
As $Q$ was chosen to be a general point of an arbitrary curve $C$ on $Y$, 
we conclude that $(Y,2cB)$ is log terminal in dimension one. 
\end{proof}

The lemma implies that the multiplier ideal $\J(Y,2cB)$
defines a zero-dimensional subscheme $\S \subset Y$. 
We have $H^1(Y,\J(Y,2cB) \otimes \O_Y(2)) = 0$ by Nadel's vanishing theorem, 
since $\om_Y$ is trivial, $B$ is cut out by forms of degree $r$, and $2cr < 2$.
It follows that there is a surjection
\[
H^0(Y,\O_Y(2)) \surj H^0(\S,\O_\S(2)) \cong H^0(\S,\O_\S)
\]
(here $\O_\S(2) \cong \O_\S$ because $\S$ is zero dimensional),
and therefore we have
\begin{equation}
\label{eq:1}
h^0(\S,\O_\S) \le h^0(Y,\O_Y(2)) = \binom{N+1}{2}.
\end{equation}

\begin{lem}
\label{l:3/2}
There exists a prime divisor $E$ over $X$ with center $P$ and log discrepancy
\[
a_E(X,cB+P) \le 0
\]
such that the center of $E$ in the blow-up of $X$ at $P$ has positive dimension. 
\end{lem}

\begin{proof}
Recall that $\mld(P;Y,cB) \le 0$. We fix a log resolution $f \colon Y' \to Y$ of
$(Y,B+P)$, and take a general hyperplane section $Z \subset Y$ through $P$. 
Let $Z' \subset Y'$ be the proper transform of $Z$. By Bertini's theorem, we can ensure 
that $Z'$ intersects transversally the exceptional locus of $f$ 
and the induced map $Z' \to Z$ is a log resolution of $(Z,B|_Z + P)$. 

We have $\mld(P;Z,cB|_Z) \le 0$ by inversion of adjunction. This means that there is 
a prime exceptional divisor $F \subset Z'$ with center $P$ in $Y$ 
and log discrepancy $a_F(Z,cB|_Z) \le 0$. 
There is a unique prime exceptional divisor $E \subset Y'$ such that $F$ is an irreducible component
of $E|_{Z'}$. Note that $E|_{Z'}$ is reduced. Since $E$ is the only prime divisor of $Y'$ that is
contained in either supports of the inverse images of $B$ and $P$
and whose restriction to $Z'$ contains $F$, we have 
$\val_E(B) = \val_F(B|_Z)$ and $\val_E(P) = \val_F(P)$.
It follows by adjunction formula that
\[
a_E(Y,cB + P) = a_F(Z,cB|_Z) \le 0.
\]
We deduce from the fact that $(Y,cB)$ is log terminal in dimension one
that the center of $E$ in $Y$ is equal to $P$. 
The fact that $E \cap Z' \ne \emptyset$
(for a general hyperplane section $Z \subset Y$ through $P$) implies that
the center of $E$ on the blow-up of $Y$ at $P$ is positive dimensional. 
\end{proof}

Let $E$ be as in Lemma~\ref{l:3/2}, and let
\[
\l := \frac{\val_E(P)}{c\val_E(B)}.
\]
In the next two lemmas, we establish opposite bounds on $\l$. 
The proof of the theorem will result by comparing the two bounds.

\begin{lem}
\label{l:2}
$\l > \dfrac{1}{N+1}$.
\end{lem}

\begin{proof}
Let $x,y \in \fm_{Y,P}$ be two general linear combinations of a given regular
system of parameters of $Y$ at $P$. 
Since the center of $E$ on the blow-up of $Y$ at $P$ is positive dimensional, by taking
$x,y$ general we can ensure that
$\val_E(f) \le \deg(f)\val_E(P)$ for any nonzero polynomial $f(x,y)$.

Let $d$ be any positive integer such that 
\[
d\val_E(P) \le - a_E(Y,2cB).
\]
For every nonzero polynomial $f(x,y)$ of degree $\le d$, we have $\val_E(f) \le -a_E(Y,2cB)$, 
and therefore $f \not\in \J(Y,2cB)\.\O_{Y,P}$. This means that if
$V \subset \O_{Y,P}$ is the $\C$-vector space spanned by the polynomials
in $x,y$ of degree $\le d$, then the quotient map $\O_{Y,P} \to \O_{\S,P}$
restricts to a injective map $V \inj \O_{\S,P}$, and therefore
\[
h^0(\S,\O_\S) \ge \dim_\C V = \binom{d+2}{2}.
\]
Comparing this inequality with the upperbound on $h^0(\S,\O_\S)$ obtained in \eqref{eq:1},
we conclude that $N > d$. It follows by our assumption on $d$ that
\[
N\val_E(P) > - a_E(Y,2cB).
\]
This means that $a_E(Y,2cB - NP) > 0$. Note, on the other hand, that
\[
a_E(Y,2cB - NP) = a_E(Y,(2-(N+1)\l)cB + P).
\]
Since $a_E(Y,cB + P) \le 0$ by Lemma~\ref{l:3/2}, we conclude that $(N+1)\l > 1$.
\end{proof}

\begin{lem}
\label{l:3}
$\l < \dfrac{\sqrt N - 2}{\sqrt{N}(N-5)}$.
\end{lem}

\begin{proof}
First, we observe that $(N-5)\l \le 1$. 
In fact, since $\lct_P(Y,P) = N-2$, we have
\[
a_E(Y,(N-3)\l c B + P) = a_E(Y,(N-2)P) \ge 0,
\]
and since $a_E(Y,cB + P) \le 0$ by Lemma~\ref{l:3/2}, we actually get $(N-3)\l \le 1$. 

Let $S \subset Y$ be a surface cut out by $N-4$ general hyperplane sections through $P$.
Note that $B|_S$ is a complete intersection zero-dimensional subscheme of $S$
cut out by two forms of degree $r$. 
We have 
\[
a_E(Y,(1-(N-5)\l)cB + (N-4)P) = a_E(Y,cB + P) \le 0.
\]
By our initial remark, the pair in the left hand side is effective. 
We can therefore apply inversion of adjunction, which gives $\mld(P;S,(1-(N-5)\l)cB|_S) \le 0$.
This means that
\[
\lct_P(S,B|_S) \le (1-(N-5)\l)c.
\]
By contrast, by using Theorem~0.1 of \cite{dFEM04}, Bezout's theorem, and the
inequality $c < 1/r$, we get the chain of inequalities
\[
\lct_P(S,B|_S) \ge \frac{2}{\sqrt{e_P(B|_S)}} \ge \frac{2}{r\sqrt N} > \frac{2c}{\sqrt N}.
\]
The lemma follows by comparing the two bounds on $\lct_P(S,B|_S)$.
\end{proof}

To conclude the proof of the theorem, 
we just observe that the inequalities in Lemmas~\ref{l:2} and~\ref{l:3}, combined,
imply that $N - 3\sqrt{N} + 1 < 0$, a condition that is never satisfied if $N \ge 7$.
\end{proof} 

We close this note with some comments on the errors in Lemmas~9.1 and~9.2 of \cite{dF13}. 

\begin{rmk}
\label{r:9.1}
Lemma~9.1 already fails in the following simple situation. 
Let $\s \colon \A^2 \to \A^1$ be the projection 
given by $\s(x,y) = x$, let $P \in \A^2$ be the origin in the coordinates $(x,y)$, 
and let $P' = \s(P)$. Let $X = (y+x^2+y^2=0) \subset \A^2$, and 
consider the divisor $E = [P]$ on $X$.
Note that $\m = \val_E(P) = 1$ and $W = W^1(E)$ is the fiber of $J_\infty X \to X$ over $P$.
Moreover, $\val_E|_{\C(\A^1)} = \val_{E'}$ where $E' = [P']$, and 
$W' = W^1(E')$ is the fiber of $J_\infty\A^1 \to \A^1$ over $P'$.
In particular, 
\[
((W')^0)_1 = (W')_1 = T_{P'}\A^1.
\] 
In the coordinates $(x,y,x',y',x'',y'',\dots)$ of $J_\infty\A^2$, 
the ideal of $W$ contains the elements $x,y,y',y''+2(x')^2$.
Fix any integer $m \ge 2$. 
Since the element $y''+2(x')^2$ is in the ideal of $W_m$,
after taking the degeneration to homogeneous ideals as in the proof of Lemma~9.1, 
the ideal of $(W_m)^0$ contains the elements $x,y,y',(x')^2$. Therefore
$((W_m)^0)_1 = \{0\} \subset T_P\A^2$ (set-theoretically), and hence
\[
\s_1\big(((W_m)^0)_1\big) = \{0\} \in T_{P'}\A^1.
\]
This shows that the lemma does not hold in this case. 

The error in the proof of Lemma~9.1 is in the last formula. The formula is true
before taking closures (namely, we have
$\pi_{m,0}^{-1}(P) \cap T_P^{(m)}\A^n = T_P^{(m)}X$),
but after taking closures in the projective space we only get an inclusion
$\ov{\pi_{m,0}^{-1}(P)} \cap \ov{T_P^{(m)}\A^n} \supset \ov{T_P^{(m)}X}$.
In the example discussed above, for instance, the point $(0:0:0:0:1)$ in the
homogeneous coordinates $(u : x' : y' : x'' : y'')$ 
belongs to the closure of $W_2$ (which is the same as
$\pi_{2,0}^{-1}(P)$), but not to the closure of $T_P^{(2)}X$.
\end{rmk}

\begin{rmk}
\label{r:9.2}
The error in the proof of Lemma~9.2 is 
in the wrong assertion that the image under a finite morphism of a Cohen--Macaulay scheme is Cohen--Macaulay. 
This fails for instance for general projections to $\P^4$ of most projective surfaces in $\P^5$. 
\end{rmk}

\end{document}